\documentclass{amsart}

\newtheorem{theorem}{Theorem}[section]
\newtheorem{lemma}[theorem]{Lemma}
\newcommand{\mychi}{\raisebox{0pt}[1ex][1ex]{$\chi$}}
\theoremstyle{definition}
\newtheorem{definition}[theorem]{Definition}
\newtheorem{example}[theorem]{Example}

\newtheorem{corollary}{Corollary}
\theoremstyle{proposition}
\newtheorem{proposition}{Proposition}
\theoremstyle{remark}
\newtheorem{remark}[theorem]{Remark}

\numberwithin{equation}{section}

\begin{document}

\title{Zero divisors and topological divisors of zero in certain Banach algebras}

%    Remove any unused author tags.

%    author one information
\author{Anurag Kumar Patel}
\address{Department of Mathematics, Banaras Hindu University, Varanasi-221005, India}
\curraddr{}
\email{anuragrajme@gmail.com}
%\thanks{The first author is supported by the Council of Scientific and Industrial Research (CSIR) NET-JRF, New Delhi, India, through grant $09/013(0891)/2019-EMR-I$} 

%    author two information
\author{Harish Chandra}
\address{Department of Mathematics, Banaras Hindu University, Varanasi-221005, India}
\curraddr{}
\email{harishc@bhu.ac.in}
\thanks{}



\date{}

\dedicatory{}

%\CorrespondingAuthor{Anurag Kumar Patel}
%
\thanks{The first author is supported by the Council of Scientific and Industrial Research (CSIR) NET-JRF, New Delhi, India, through grant $09/013(0891)/2019-EMR-I$}

\begin{abstract}
		In this paper we prove that an element $f\in \mathcal{A}(\mathbb{D})$ is a topological divisor of zero(TDZ) if and only if there exists $z_0 \in \mathbb{T}$ such that $f(z_0)=0.$ We also give a characterization of TDZ in the Banach algebra $L^\infty(\mu).$  Further, we prove that the multiplication operator $M_h$  is a TDZ in $\mathcal{B}(L^p(\mu))~(1\leq p\leq\infty)$ if and only if $h$ is a TDZ in $L^\infty(\mu).$ Subsequently, we show that a composition operator $C_{\phi}$ is a TDZ in $\mathcal{B}(L^2(\mu))$ if and only if $\frac{d\mu \phi^{-1}}{d\mu}$ is a TDZ in $L^{\infty}(\mu).$ Lastly, we determine composition operators on the Hardy spaces $\mathbb{H}^p(\mathbb{D})$ and $\ell^p$ spaces which are zero-divisors.
\end{abstract}

\maketitle
\section{Introduction}
The notion of zero-divisor{\cite{GelfSilov, Simmons}} is well known from ring theory. In a ring $\mathfrak{R},$ an element $z\in \mathfrak{R}$ is called a left~$(right)$ zero-divisor in $\mathfrak{R}$ if $zx=0~(xz=0)$ for some $0\neq x\in \mathfrak{R}.$ An element $z\in \mathfrak{R}$ is called zero-divisor if it is either a left (or right) zero-divisor.
The concept of topological divisor of zero(\boldmath{TDZ}){\cite{Conway, Simmons}} is an important generalization in a Banach algebra, introduced by Shilov in {\cite{Shilow}},  with significant later contributions by Zelazko, Kaplansky et al. Let $\mathfrak{B}$ be a Banach algebra. An element $x\in \mathfrak{B}$ is a TDZ in $\mathfrak{B}$ if $$xx_n \to 0 \text{~or~} x_nx\to 0$$ for some sequence $\{x_n\}_{n=1}^\infty$ of elements of $\mathfrak{B}$ satisfying $$\|x_n\|=1,\quad n=1,2,...$$ 
\\
Bhatt and Dedania in {\cite{Bhatt}} gave a classification of Banach algebras in which every element is a TDZ.
The purpose of the present paper is to determine TDZ in some Banach algebras which do not fall in the preceding category.
\\ 	
In this paper, $\mathbb{N}$ denotes the set of all positive integers, $\mathbb{C}$ the set of complex numbers, $\vert A\vert$ the cardinality of a set $A $ and $\mathcal{R}(f)$ denotes the range of a function $f$. Let $\mathbb{D}=\{z\in \mathbb{C}:\vert z\vert<1\},$ $\mathbb{\overline{D}}= \{z\in \mathbb{C}:\vert z\vert\leq 1\}$ and $\mathbb{T}=\{z\in \mathbb{C}:\vert z\vert=1\}$ be the unit circle. Let $\mathcal{A}(\mathbb{D})$ denote the Banach algebra of all continuous functions on $\mathbb{\overline{D}}$ which are analytic in $\mathbb{D}$ under the sup-norm and $L^\infty(\mu)$ be the Banach algebra of all essentially bounded complex-valued measurable functions on $X$ under the essential supremum norm. Further, $\sigma_{p} (M_h),~\sigma_{c} (M_h)$ and $\sigma (M_h)$ denotes the point spectrum, the continuous spectrum and the spectrum of $M_h$ respectively. Further, $\mathcal{B}(L^p(\mu))~(1\leq p\leq\infty)$ denotes the Banach algebra of all bounded linear operators on $L^p(\mu).$ \\ 
The current article is aimed at finding zero-divisors and TDZ in $\mathcal{A}(\mathbb{D})$ and $L^\infty(\mu)$ in its full generality. It also determines the conditions under which the multiplication operator on $L^p(\mu)~(1\leq p\leq\infty)$ and the composition operator on $L^2(\mu)$ are TDZ. The paper is structured as follows.\\
In Section \ref{sec2}, we cite certain definitions and results which will be used in this paper. Section \ref{sec3.1} is devoted for the characterization of TDZ in $\mathcal{A}(\mathbb{D}).$ In section \ref{sec3.2}, we characterize the zero-divisors and TDZ in $L^\infty(\mu).$ Section \ref{sec3.3} is devoted to find a necessary and sufficient condition for a multiplication operator $M_h$ to be TDZ in $\mathcal{B}(L^p(\mu)).$ In Section \ref{sec3.4}, we find a necessary and sufficient condition for the composition operator $C_\phi$ on $L^2(\mu)$ to be a TDZ in $\mathcal{B}(L^2(\mu)).$ In Section \ref{sec3.5} we classify the composition operators on $\mathbb{H}^p(\mathbb{D})$ which are zero-divisors in $\mathcal{B}(\mathbb{H}^p(\mathbb{D})).$ Finally, in Section \ref{sec3.6}, we determine the composition operators on $\ell^p$ which are zero-divisors and TDZ in $\mathcal{B}(\ell^p).$

\section{Preliminaries} \label{sec2}
In this section, we state few definitions and some known results which will be used later.		
\begin{definition}\rm{\cite{RGDoug,Simmons}}\label{anurag1}
	Let $\mathfrak{B}$ denote a Banach algebra with identity $1$. An element $x\in \mathfrak{B}$ is called regular if there exists $y\in \mathfrak{B}$ such that $ xy=yx=1.$ Otherwise, $x\in \mathfrak{B}$ is  called a singular element. 
\end{definition}
%\begin{definition}\rm{\cite{GelfSilov,Simmons}} \label{anurag2}
%	Let $\mathfrak{R}$ be a ring. We call $z\in \mathfrak{R}$ a left $(right)$ zero-divisor if there is an $x\in \mathfrak{R}$ such that $x\neq0$ and $zx=0~(xz=0)$ respectively. An element $z\in \mathfrak{R}$ is called zero-divisor if it is either a left zero-divisor or a right zero-divisor.
%\end{definition}
%\begin{definition}\rm{\cite{Conway, Simmons}} \label{anurag3}
%	An element $z$ in a Banach algebra $\mathfrak{B}$ is called a TDZ if there exists a sequence $\{z_n\}_{n=1}^\infty$ in $\mathfrak{B}$ satisfying the following conditions. \begin{enumerate}
	%		\item[$(1)$] $\|z_n\|=1 \quad \forall\ n \in \mathbb{N}$;
	%		\item[$(2)$] Either $zz_n \to 0$ or $z_nz\to 0.$
	%	\end{enumerate}
%\end{definition}

\begin{theorem}\rm {\cite{Angus}}\label{anurag4}
	In a Banach algebra, the set of all TDZ is a closed set.
\end{theorem}

\begin{definition}\rm {\cite{Daniel Alpay,Peter Duren,Hoffman}} \label{Hardy space} For $0<p<\infty$ the Hardy space $\mathbb{H}^p(\mathbb{D})$ is the set of functions analytic on the unit disk for which
	\begin{equation*}
		\sup_{r\in(0,1)}\frac{1}{2\pi}\int_{0}^{2\pi}|f(re^{i\theta})|^pd\theta<\infty.
	\end{equation*}
	Let $\mathbb{H}^\infty(\mathbb{D})$ be the Banach algebra of all bounded analytic functions in $\mathbb{D}$ under the sup-norm. 
\end{definition}

%\begin{definition}\rm {\cite{Daniel Alpay,Peter Duren,Hoffman}} \label{Hardy space} The Hilbert space $\mathbb{H}^2(\mathbb{D})$ is the set of functions analytic in $\mathbb{D}$ and such that
%	\begin{equation*}
	%		\sup_{r\in(0,1)}\frac{1}{2\pi}\int_{0}^{2\pi}|f(re^{it})|^2dt<\infty,
	%	\end{equation*} or equivalently, such that 
%	\begin{equation*}
	%		\sum_{n=0}^{\infty}|a_n|^2<\infty,
	%	\end{equation*} holds, where $f(z)=\sum_{n=0}^{\infty}a_nz^n$ is the power series of $f$ at the origin.
%\end{definition}

%		\begin{Theorem}
	%			
	%		\end{Theorem}

\begin{definition}\rm{\cite{Carl C Cowen,John B Garnett,HJ Schwartz}}
	Let $\phi$ be a holomorphic function on $\mathbb{D}$ with $\phi(\mathbb{D})\subset \mathbb{D}.$ Then the equation $C_{\phi}f=f \circ \phi$ defines a composition operator $C_{\phi}$ on the space of holomorphic functions in $\mathbb{D}.$ The Littlewood's subordination principle{\cite{Peter Duren,JELittlewood}} says that $C_{\phi}$ is bounded linear operator on $\mathbb{H}^p(\mathbb{D})~(1\leq p\leq\infty).$
\end{definition}

\begin{theorem} \rm{\cite{Caughran}} \label{Caratheodory}
	Let $\phi:\mathbb{D} \to\mathbb{D}$ be an analytic map which maps $\mathbb{D}$ conformally on to a Carath$\acute{e}$odory domain, then the $\mathcal{R}(C_{\phi})$ is dense in $\mathbb{H}^{p}(\mathbb{D}),~1\leq p<\infty.$
\end{theorem}

\begin{theorem}\rm({\cite{R.C.Roan,D.Sarason}}) \label{weak*generator}
	Let	$\phi:\mathbb{D} \to\mathbb{D}$ be an analytic map and $\phi$ is a weak* generator of $H^\infty$, then the $\mathcal{R}(C_{\phi})$ is dense in $\mathbb{H}^{p}(\mathbb{D}),~0<p<\infty.$
\end{theorem}

\begin{definition}\rm({\cite{Zimmer}})\label{anurag5}
	Let $(X,\Omega,\mu)$ be a measure space and $Y$ denote a topological space. Let $f:X\to Y$ be a measurable function. An element $y\in Y$ is in the essential range of $f $ if for any neighborhood $U$ of $y$, $\mu(f^{-1}(U))>0.$ The set of all such elements is called the essential range of $f$ and is denoted by ess.range$(f).$ 
\end{definition}
\begin{definition}\rm{\cite{Bollobas,Sheldom,Royden}}
	Suppose $(X,\Omega,\mu)$ is a $\sigma-$finite measure space and $1\leq p<\infty.$ Let $L^p(\mu)=\{f:X\to \mathbb{C} : f \textrm{ is measurable and } \int_{X}\vert f\vert^{p}d\mu<\infty \}$ and $L^{\infty}(\mu)=\{f:X\to \mathbb{C} : f \textrm{ is measurable and ess.sup. } \vert f\vert<\infty\}.$ Note that
	\begin{enumerate}
		\item[(i)] For $1\leq p<\infty,~L^p(\mu)$ is a Banach space under the $L^p-$norm defined as $$\vert\vert f\vert\vert_p=(\int \vert f\vert^p d\mu)^\frac{1}{p}~\forall~f\in L^p(\mu).$$ 
		\item[(ii)] $L^{\infty}(\mu)$ is a Banach algebra with identity under the norm defined as $$\|f\|=\|f\|_\infty=\inf \{K:\mu\{t:\vert f(t)\vert>K\}=0\}~\forall~f\in L^{\infty}(\mu).$$   
	\end{enumerate} 
\end{definition}
\begin{lemma}\rm  {\cite{Sheldom,RGDoug,Zimmer}}\label{anurag6}
	Suppose $(X,\Omega,\mu)$ is a measure space and $h\in L^\infty(\mu).$ For $1\leq p\leq\infty,$ define the multiplication operator  $M_h:L^p(\mu)\to L^p(\mu)$ by $M_h(f)(x)=h(x) f(x)~\forall~ x\in X$. Then $M_h\in \mathcal{B}(L^p(\mu))$ and 
	\begin{enumerate}
		\item [(i)] $\textrm{ess.range}(h)=\sigma (M_h)=\{\lambda \in \mathbb{C}:\forall~\epsilon>0,\mu (\{x\in X: \vert h(x)-\lambda\vert<\epsilon\})>0 \}$
		\item [(ii)] $\sigma_{p}(M_h)=\{\lambda \in \mathbb{C}:\mu(\{x\in X:h(x)=\lambda\})>0\}.$
	\end{enumerate}
	
\end{lemma}
\begin{lemma}\rm {\cite{Anurag}}\label{anurag11}
	Let $\mathfrak{B}$ be a commutative Banach algebra and $x\in \mathfrak{B}$ be a TDZ. Then for each $y\in \mathfrak{B},\ xy$ is also a TDZ.
\end{lemma}

\begin{theorem}\rm {\cite{Anurag}}\label{anurag10}	
	Let $\mathcal{A}=\mathcal{A}(\mathbb{D})$ be the disk algebra. Let $f(z)=\left(\frac{z-z_0}{2}\right)$ for some $z_0 \in \mathbb{C}$. Then $f$ is TDZ in $\mathcal{A}$ if and only if $\vert z_0\vert=1.$
\end{theorem}

%\begin{theorem}\rm {\cite{Walter}} [Mergelyan's Theorem] \label{Mergelyan}
%	If $K$ is a compact set in the plane whose complement is connected, if $f$ is continuous complex function on $K,$ and if $\epsilon>0,$ then there exists a polynomial $p$ such that $|f(z)-p(z)|<\epsilon$ for all $z\in K.$
%\end{theorem}

\begin{definition}\rm {\cite{RKM,RK}}
	If $(X,\Omega,\mu)$ is a $\sigma-\text{finite}$ measure space and $\phi$ is a non-singular measurable transformation on $X$ into itself, then the composition transformation $C_\phi$ on $L^p(\mu)~(1\leq p<\infty)$ is defined by the following equation.$$ C_\phi f=f\circ\phi~\text{for every}~f\in L^p(\mu).$$ 
	If $C_\phi \in \mathcal{B}(L^p(\mu)),$ the Banach algebra of all bounded linear operators on $L^p(\mu),$ then it is called a composition operator induced by $\phi.$\\
	\begin{theorem}\rm {\cite{RKM}}
		A measurable transformation $\phi:X \to X$  defines a composition operator on $L^p(\mu)~(1\leq p<\infty)$ if and only if $\mu\phi^{-1}\ll \mu$ and $\frac{d\mu \phi^{-1}}{d\mu}\in L^\infty(\mu),$ where $\frac{d\mu \phi^{-1}}{d\mu}$ is Radon-Nikodym derivative of $\mu \phi^{-1}$ with respect to $\mu.$ In this case \begin{align*}
			\|C_\phi\|=\|\frac{d\mu \phi^{-1}}{d\mu}\|_{\infty}^{\frac{1}{p}}.
		\end{align*}
	\end{theorem}
	When $X=\mathbb{N}$	and $\mu$ is the counting measure on $\mathbb{N},$ then $L^p(\mu)$ becomes $\ell^p.$ In this setting we have the following result.
\end{definition}
\begin{theorem}\rm {\cite{RK}}
	Let $\phi:\mathbb{N}\to \mathbb{N}$ be a self map. Then the composition transformation $C_\phi$ induced by $\phi$ on $\ell^p~(1\leq p<\infty)$ and defined as $$C_\phi(f)=\sum_{n=1}^{\infty}f(n)\mychi_{\phi^{-1}(n)}$$ is a bounded  linear operator if and only if the set $\{|\phi^{-1}(n)|: n\in \mathbb{N}\}$ is bounded.
\end{theorem}
\begin{theorem}\rm {\cite{RKM}} \label{RN derivative}
	Let $\phi:X \to X$ be a measurable transformation and let $C_{\phi}$ be the composition operator induced by $\phi$ on $L^2(\mu).$ Then $C_{\phi}^*C_{\phi}=M_\frac{d\mu \phi^{-1}}{d\mu},$ where $\frac{d\mu \phi^{-1}}{d\mu}$ is Radon-Nikodym derivative of $\mu \phi^{-1}$ with respect to $\mu.$
\end{theorem}

\begin{theorem}\rm {\cite{RK,RKM}}\label{anurag7}
	Let $\phi$ be a self map on $\mathbb{N}$ and $C_{\phi}$ be the composition operator induced by $\phi$ on $\ell^p~(1\leq p<\infty).$ Then,\begin{enumerate}
		\item [$(1)$] $C_{\phi}$ is one-one if and only if $\phi$ is onto;
		\item [$(2)$] $C_{\phi}$ is onto if and only if $\phi$ is one-one;
		\item [$(3)$] $C_{\phi}$ is invertible if and only if $\phi$ is invertible;
	\end{enumerate} 
\end{theorem}
\section{Main results} \label{sec3}
\subsection{TDZ in $\mathcal{A}(\mathbb{D})$} \label{sec3.1}\

In the following theorem we characterize the TDZ in $\mathcal{A}(\mathbb{D}).$
\begin{theorem}\label{TDZ in A(D)}
	An element $f\in \mathcal{A}(\mathbb{D})$ is a TDZ if and only if there exists $z_0 \in \mathbb{T}$ such that $f(z_0)=0.$
\end{theorem}
\begin{proof}
	Let	$f \in \mathcal{A}(\mathbb{D})$ be a TDZ. Then there exists a sequence $\{f_n\}_{n=1}^\infty$ in $\mathcal{A}(\mathbb{D})$ such that $\Vert f_n\Vert=1 \ \forall \ n\in \mathbb{N}$, and $\Vert ff_n\Vert \to 0.$
	Suppose that $f(z)\neq 0$ for all $z\in \mathbb{T}.$ Then $\vert f\vert>0$ on $\mathbb{T}.$ Now the continuity of $\vert f\vert $ on the compact set $\mathbb{T}$   implies that $\vert f\vert$ attains its maximum and minimum values on $\mathbb{T}.$ Therefore, there exists a $\delta >0$ such that $$\vert f(z)\vert >\delta ~\forall z \in \mathbb{T}.$$ Further, for each $n\in \mathbb{N}$ $$\Vert ff_n\Vert=\max_{z\in \mathbb{T}}\vert f(z)f_n(z)\vert>\delta \cdot \max_{z\in \mathbb{T}}\vert f_n(z)\vert=\delta \cdot \Vert f_n\Vert=\delta>0.$$ This contradicts the fact $\Vert ff_n\Vert \to 0.$ Hence there exists a point $z_0 \in \mathbb{T}$ such that $f(z_0)=0.$\\
	Conversely, assume that there exists $z_0 \in \mathbb{T}$ such that $f(z_0)=0.$ Since $f$ is continuous on $\mathbb{\overline{D}}$ and analytic on $\mathbb{D},$ by Mergelyan's theorem\cite{Walter}, there exists a sequence of polynomials $\{p_n\}_{n=1}^{\infty}$ that uniformly converges to $f.$ Let $q_n(z)=p_n(z)-p_n(z_0).$ Then $\{q_n\}_{n=1}^{\infty}$ uniformly converges to $f$ and $q_n(z_0)=0.$ Hence $q_n(z)=(z-z_0)r_n(z)~\forall~n\geq 1.$ Now, since $z_0\in \mathbb{T},$ it follows from the Theorem $\ref{anurag10}$ that $h(z)=(z-z_0)$ is a TDZ. Consequently, by Lemma \ref{anurag11}, $q_n(z)$ is also a TDZ for each $n\geq1$. Since $\{q_n\}_{n=1}^{\infty}$ converges to $f,$ hence by Theorem \ref{anurag4}, $f$ is also a TDZ. 
\end{proof}
\begin{remark}
	We realise that an alternative proof of the above Theorem \ref{TDZ in A(D)} can be obtained using the notion of Shilov boundary and applying the result [\cite{E Kaniuth}, Corollary 3.4.6, page 171].
\end{remark}

%		In this section, we give a characterization of zero-divisors and topological divisors of zero in the Banach algebra $L^{\infty}(\mu).$ 

\subsection{TDZ in $L^{\infty}(\mu)$} \label{sec3.2}\

The following theorem characterizes the zero-divisors in $L^{\infty}(\mu).$
\begin{theorem}
	Let $f\in L^\infty(\mu)$. Then $f$ is a zero-divisor if and only if $\mu(E^c)>0,$ where $E=\{x\in X : f(x)\neq0\}.$
\end{theorem}
\begin{proof}
	Suppose $E=\{x\in X : f(x)\neq0\}$ and $\mu(E^c)>0.$ Let $g=\mychi_{E^c}.$ Then, $g\neq0$ and $f\cdot g=0.$ This implies that $f$ is a zero- divisor.\\
	Conversely, suppose that $f$ is a zero-divisor. Then there exists a $0\neq g\in L^\infty(\mu)$ such that $f\cdot g=0.$ Let $E=\{x\in X :f(x)\neq0\}.$ Then $E$ is a measurable set. Also, $f\cdot g=0$ implies $E\subset\{x\in X:g(x)=0\}$. This implies that $\{x\in X:g(x)\neq 0\}\subset E^c.$ Since $g\neq 0,$ then $\mu(\{x\in X:g(x)\neq 0\})>0.$ Hence $\mu(E^c)>0.$ 
\end{proof}
The next theorem determines the topological divisors of zero in $\L^\infty(\mu).$
\begin{theorem}\label{anurag8}
	Let $f\in L^\infty(\mu).$ Then $f$ is a TDZ if and only if there exists a sequence $\{E_n\}_{n=1}^{\infty}$ of measurable sets in $X$ with $\mu(E_n)>0$ for each $n\geq 1$ and $\Vert f\vert_{E_n} \Vert \to 0.$
\end{theorem}
\begin{proof}
	Suppose there exists a sequence $\{E_n\}_{n=1}^{\infty}$ of measurable sets with $\mu (E_n)>0$ for each $n\geq 1$ and $\Vert f\vert_{E_n}\Vert \to 0$ as $n \to \infty.$
	Let $g_n=\mychi _{E_n}.$ Then $\Vert g_n\Vert=1$ for each $n \geq 1$ and $fg_n=f\vert_{E_n}.$ Hence $\Vert fg_n\Vert=\Vert f\vert_{E_n}\Vert \to 0$ as $n \to \infty.$ This implies that $f$ is a TDZ.\\
	Conversely, suppose $f\in L^\infty(\mu)$ is a TDZ. Then there exists a sequence $\{g_n\}_{n=1}^{\infty}$ in $L^\infty(\mu)$ with $\Vert g_n\Vert=1$ for each $n\geq1$ and $\Vert fg_n\Vert \to 0.$\\ Let $E_n=\{x\in X:\vert g_n(x)\vert>1-\frac{1}{n}\}.$ Since $\|g_n\|=1,$ hence  $\mu(E_n)>0.$ 
	Clearly, $\Vert fg_n\Vert \to 0$ implies $\Vert fg_n\vert_{E_n}\Vert \to 0$ as $n \to \infty.$\\
	Further, for $x\in E_n$ $$\vert fg_n(x)\vert=\vert f(x)\vert \vert g_n(x)\vert> \vert f(x)\vert(1-\frac{1}{n}).$$ Consequently $$\vert f(x)\vert< \frac{1}{(1-\frac{1}{n})}\vert fg_n(x) \vert \ \forall\ x\in E_n.$$  Hence $$\Vert f\vert_{E_n}\Vert\leq \frac{1}{(1-\frac{1}{n})} \Vert fg_n\vert_{E_n}\Vert \to 0.$$ Thus $$\Vert f\vert_{E_n}\Vert \to 0.$$ 
\end{proof}	
The following Remark can be easily obtained.
\begin{remark}
	An element $f\in L^\infty(\mu)$ is regular if and only if there exists a $\lambda_0>0$ such that $\vert f\vert\geq \lambda_0$ a.e. on $X.$
\end{remark}
%		\begin{proof}
	%			Suppose $\vert f\vert\geq \lambda_0$ a.e. on $X$ and let $g=\frac{1}{f}$ a.e. on $X.$ Then $g\in L^\infty(\mu)$ as $\vert g\vert\leq \frac{1}{\lambda_0}$ a.e. on $X.$ Further, $gf=fg=1$ a.e. on $X.$ Therefore $f$ is a regular element in $L^\infty(\mu)$ .
	%			Conversely, assume that $f\in L^\infty(\mu)$ is a regular element. Then there exists a $g\in L^\infty(\mu)$ such that $fg=gf=1$ a.e. on $X.$ This implies $\vert g\vert >0$ a.e. on $X.$ Therefore, $\Vert g\Vert>0,$ and $\vert f\vert =\frac{1}{\vert g\vert} \geq \frac{1}{\Vert g\Vert}>0$ a.e.. Choosing $\lambda_0=\frac{1}{\|g\|}>0$ , we get $\vert f\vert\geq\lambda_0$ a.e. on $X.$ This completes the proof.
	%		\end{proof}

\begin{corollary}\label{anurag9}
	
	In $L^\infty(\mu),$ The collection of all TDZ is equivalent to the collection of all singular elements. 
\end{corollary}
\begin{proof}
	Clearly, a TDZ is a singular element.\\
	Let $f\in L^\infty(\mu)$ be a singular element. Then $f$ is not regular in $L^\infty(\mu).$ Therefore, for each $n\in \mathbb{N},$ there exists a measurable set $E_n$ with $0<\mu(E_n)<\infty$ such that $\vert f\vert_{E_n}\vert<\frac{1}{n}.$ Hence $\Vert f\vert_{E_n}\Vert\leq\frac{1}{n} \ \forall\ n\geq 1.$ Therefore $\Vert f\vert_{E_n}\Vert \to 0$ as $n \to \infty.$\\
	Consequently, by Theorem $\ref{anurag8},$ $f$ is a TDZ.
\end{proof} 

\begin{remark}
	We note that the Corollary \ref{anurag9} also follows from the result  that any non-invertible element in a symmetric normed algebra without radicals is a TDZ. \rm(See {\cite{GelfSilov}}, Corollary $2$, sec.$10.2$).
\end{remark}

\begin{theorem}
	Let	$h \in L^\infty(\mu)$ and $M_h$ be the multiplication operator on $L^\infty(\mu)$ induced by $h.$ Then $h$ is a TDZ if and only if $0\in \textrm{ess.range}(h)=\sigma (M_h).$ 
\end{theorem}
\begin{proof}
	Let $0\in ess.range(h)=\sigma (M_h)=\{\lambda \in \mathbb{C}:\forall \epsilon>0,\mu (\{x\in X: \vert h(x)-\lambda\vert<\epsilon\})>0 \}.$ Then for each $n\geq 1,$ $\mu (\{x\in X:\vert h(x)\vert<\frac{1}{n}\})>0.$ Let $E_n=\{x\in X:\vert h(x)\vert<\frac{1}{n}\}.$ Thus for each $n\geq 1,~  \mu(E_n)>0$  and $\vert h\vert_{E_n}\vert<\frac{1}{n}.$ This implies that $\Vert h\vert_{E_n}\Vert \leq \frac{1}{n} \ \forall \ n\in\mathbb{N}.$ Therefore $\Vert h\vert_{E_n} \Vert \to 0$ as $n \to \infty.$ Hence $h$ is a TDZ.\\
	Conversely, suppose $h$ is a TDZ. Then, by Corollary $\ref{anurag9},$ $h$ is not a regular element. Therefore, for each $\epsilon>0,~\mu (\{x\in X : \vert h(x)-0\vert<\epsilon\})>0.$ This implies that  $0\in$ ess.range$(h)=\sigma (M_h).$ 
\end{proof}
\begin{remark}
	Let $h \in L^\infty(\mu)$ be a TDZ. Then
	\begin{enumerate}
		\item $0\in \sigma_{p}(M_h)$ if and only if $\mu(\{x\in X:h(x)=0\})>0.$
		\item $0\in \sigma_{c}(M_h)$ if and only if $\mu(\{x\in X:h(x)=0\})=0$ and $\mu(\{x\in X:\vert h(x)\vert< \epsilon\})>0$ for each $\epsilon>0.$ 
	\end{enumerate}
\end{remark}

\begin{theorem}
	Let $h\in L^\infty(\mu)$ and $M_h:L^\infty(\mu)\to L^\infty(\mu)$ be the multiplication operator. Then $h$ is a zero-divisor if and only if $0\in \sigma_{p} (M_h).$ 
\end{theorem}
\begin{proof}
	Let $0\in \sigma_{p} (M_h).$ Then $\exists~ 0\neq f \in L^\infty(\mu)$ such that $M_h(f)=0\cdot f=0.$ This implies that $h \cdot f=0.$ Hence $h$ is a zero-divisor in $L^\infty(\mu)$ .\\
	Conversely, suppose $h \in L^\infty(\mu)$ is a zero-divisor. Then $\exists~0\neq f \in L^\infty(\mu)$ such that $h \cdot f=0=0\cdot f$ a.e.. This implies that $M_h(f)=0\cdot f.$ Hence $0\in \sigma_{p} (M_h).$   
\end{proof}

\subsection{Multiplication operator as TDZ} \label{sec3.3}\

In this section, we state and prove the necessary and sufficient condition for the multiplication operator $M_h$  to be TDZ in $\mathcal{B}(L^p(\mu)).$

\begin{theorem}\label{Anurag6}
	Let $h \in L^\infty(\mu)$ and $1\leq p\leq\infty.$ Define the multiplication operator $M_h:L^p(\mu) \to L^p(\mu)$ by $M_h(f)=h \cdot f=h f$. Then  $M_h$ is a TDZ in $\mathcal{B}(L^p(\mu))$ if and only if $h$ is a TDZ in $L^\infty(\mu).$ 
\end{theorem}
\begin{proof}
	Suppose $h \in L^\infty(\mu)$ be a TDZ. Let $\{h_n\}_{n=1}^{\infty}$ be a sequence in $L^\infty(\mu)$ such that $$\Vert h_n\Vert=1 \ \forall \  n\in \mathbb{N}~ and ~\Vert h h_n\Vert \to 0.$$\\
	For each $n\geq1,$ define $M_{h_n}:L^p(\mu) \to L^p(\mu)$ as $M_{h_n}(f)=h_n f.$ Then $$\Vert M_{h_n}\Vert=\Vert h_n\Vert=1 \ \forall \ n\in \mathbb{N}~ and ~M_hM_{h_n}(f)=h h_n f.$$ Hence $\Vert M_hM_{h_n}\Vert=\Vert h h_n\Vert \to 0.$  This shows that $M_h$ is a TDZ in $\mathcal{B}(L^p(\mu))$. \\
	For the converse, suppose $h \in L^\infty(\mu)$ is not a TDZ in $L^\infty(\mu)$. Then by Corollary\ref{anurag9}, $h$ is a regular element in $L^\infty(\mu)$. Therefore there exists a $g \in L^\infty(\mu)$ such that $h \cdot g=g \cdot h=1.$  This implies that $M_{h g}=\text{I}.$ Hence $M_{h g}=M_{h}M_{g}=M_{g}M_{h}=\text{I}.$ This shows that $M_h$ is regular element in $\mathcal{B}(L^p(\mu))$. Hence $M_h$ can not be a TDZ.\\
	
\end{proof}
\subsection{Composition operator on $L^2(\mu)$ as TDZ} \label{sec3.4}

\begin{theorem} \label{Anurag7}
	Let $H$ be a Hilbert space and $T\in \mathcal{B}(H).$ Then $T$ is a TDZ in $\mathcal{B}(H)$ if and only if $T^*T$ is a TDZ in $\mathcal{B}(H).$
\end{theorem}\label{Anurag8}
\begin{proof}
	Let $T\in \mathcal{B}(H)$ be a TDZ. Hence there exists a sequence $\{T_n\}_{n=1}^\infty$ in $\mathcal{B}(H)$ such that $\|T_n\|=1~\forall~n\in \mathbb{N}$ and $\|TT_n\| \to 0.$ Since $\|T_n\|=\|T^*_n\|=1~\forall~n\in \mathbb{N}$ and $\|T^*\|=\|T\|,$ consequently $$\|(T^*T)T_n\|\leq \|T^*\|\|TT_n\|\to 0,$$ and $$\|T^*_n(T^*T)\|\leq\|T^*_nT^*\|\|T\|=\|TT_n\|\|T\| \to 0.$$  Hence $T^*T$ is a TDZ.\\
	Conversely, suppose $T^*T$ is a TDZ in $\mathcal{B}(H).$ Then there exists a sequence $\{T_n\}_{n=1}^\infty$ in $\mathcal{B}(H)$ such that $\|T_n\|=1~\forall~n\in \mathbb{N}$ and $\|T^*TT_n\|\to 0.$ Observe that, for each $x\in H$ $$\|TT_nx\|^2=\langle T^*TT_nx,\; T_nx \rangle \leq \|T^*TT_nx\|\|T_nx\|.$$ Hence $$\|TT_n\|^2 \leq \|T^*TT_n\|\|T_n\|=\|T^*TT_n\| \to 0.$$
	%			Therefore $\|TT_n\| \to 0$ as $n \to \infty.$ 
	Thus, $T$ is TDZ.
\end{proof}	
The following Remark can be readily obtained.
\begin{lemma}\label{Anurag9}
	Let $H$ be a Hilbert space and $T\in \mathcal{B}(H).$ Then $T$ is a left(right) TDZ in $\mathcal{B}(H)$ if and only if  $T^*$ is a right(left) TDZ in $\mathcal{B}(H).$
\end{lemma}
\begin{proof}
	Let $T\in \mathcal{B}(H)$ be a left TDZ. Then there exists a sequence $\{T_n\}_{n=1}^\infty$ in $\mathcal{B}(H)$ such that $\|T_n\|=1~\forall~n\in \mathbb{N}$ and $\|TT_n\| \to 0.$ Since $\|T_n\|=\|T^*_n\|=1~\forall~n\in \mathbb{N}$ and $$\|(T^*_nT^*)\|=\|(TT_n)^*\|=\|TT_n\|\to 0.$$ Hence $T^*$ is a right TDZ.\\
	Conversely, suppose $T\in \mathcal{B}(H)$ and let $T^*$ be a right TDZ. Then, from the above argument, $(T^*)^*=T$ is a left TDZ in $\mathcal{B}(H).$
\end{proof}
\begin{theorem}\label{Anurag10}
	Composition operator $C_\phi$ on $L^2(\mu)$ is a TDZ in $\mathcal{B}(L^2(\mu))$ if and only if we can find a sequence $\{E_n\}_{n=1}^{\infty}$ of measurable sets in $X$ such that $\mu(E_n)>0$ for each $n\geq 1$ and $\|\frac{d\mu \phi^{-1}}{d\mu}|_{E_n} \| \to 0.$
\end{theorem}
\begin{proof}
	Let $C_{\phi}\in \mathcal{B}(L^2(\mu))$ be a TDZ. Then by Theorem \ref{Anurag7}, $C_{\phi}^*C_{\phi}$ is also a TDZ. Further, by Theorem \ref{RN derivative}, $$C_{\phi}^*C_{\phi}=M_\frac{d\mu \phi^{-1}}{d\mu}.$$ Hence, by Theorem \ref{Anurag6}, $\frac{d\mu \phi^{-1}}{d\mu}$ is a TDZ in $L^\infty(\mu)$. Consequently, by Theorem \ref{anurag8}, we can find a sequence $\{E_n\}_{n=1}^{\infty}$ of measurable sets in $X$ with $\mu(E_n)>0$ for each $n\geq 1$ and $\|\frac{d\mu \phi^{-1}}{d\mu}|_{E_n} \| \to 0.$  \\
	Conversely, suppose that we have a sequence $\{E_n\}_{n=1}^{\infty}$ of measurable sets in $X$ such that $\mu(E_n)>0$ for each $n\geq 1$ and $\|\frac{d\mu \phi^{-1}}{d\mu}|_{E_n} \| \to 0.$ Then, by Theorem \ref{anurag8}, $\frac{d\mu \phi^{-1}}{d\mu}$ is a TDZ in $L^\infty(\mu)$. Now, combining  Theorem \ref{Anurag6} and Theorem \ref{Anurag7}, the desired conclusion follows.
\end{proof}	

\begin{example}
	Let $\mu$ denote the counting measure on $\mathbb{N}$ and $\phi: \mathbb{N} \to \mathbb{N}$ be defined as $\phi(n)=n+1.$ Let $C_\phi$ denote the composition operator on $\ell^2$ induced by $\phi.$ It is easy to see that $$\frac{d\mu \phi^{-1}}{d\mu}(n)=\begin{cases} {0}, & n=1, \\ 1, & n>1. \end{cases}$$
	Hence, by above theorem, $C_\phi$ is a TDZ in $\mathcal{B}(\ell^2).$
\end{example}

\subsection{Composition operator on $\mathbb{H}^p(\mathbb{D})~~(1\leq p<\infty)$ as zero-divisor} \label{sec3.5}\

In this section, we give a characterization of the composition operators on the Hardy space $\mathbb{H}^p(\mathbb{D})$ which are zero-divisors in $\mathcal{B}(\mathbb{H}^p(\mathbb{D})).$ We begin with the following theorem.\\
\begin{proposition}\label{9}
	Let $X$ be a Banach space and $T\in \mathcal{B}(X).$ Then $T$ is a right zero-divisor if and only if $\mathcal{R}(T)$ is not dense in $X$.
\end{proposition}
\begin{proof}
	Suppose $T\in \mathcal{B}(X)$ and $\mathcal{R}(T)$ is not dense in $X.$ Then by Hahn-Banach theorem, there exists a non-zero bounded linear functional $f$ on $X$ such that $f(y)=0$ for all $y\in \overline{T(X)}.$ For any nonzero vector $x_0\in X,$ define $S:=x_0\bigotimes f:X\to X$ as
	$S(x)=f(x)x_0.$ Then, $S$ is a nonzero element of $\mathcal{B}(X)$ such that
	
	\begin{align*}
		(ST)(x)&=f(T(x))x_0\\& =0~~\text{for all}~ x\in X.
	\end{align*}
	Hence $T$ is a right zero-divisor.\\
	%	Let $M=\overline{\mathcal{R}(T)}$ and $M^{\perp}$ be its orthogonal complement in $H.$ Then $X=M\bigoplus M^\perp.$ Let $P_{M^\perp}$ be the projection map on $M^\perp$ along $M.$ Then for each $x \in H,~P_{M^\perp}(Tx)=0.$ Hence $T$ is a right zero-divisor.\\
	Conversely, suppose $\mathcal{R}(T)$ is dense in $X.$ Let $S \in \mathcal{B}(X)$ be such that $S(Tx)=0~\forall~x\in X.$ Then $S\equiv 0$ on $\mathcal{R}(T).$ Since $\mathcal{R}(T)$ is dense in $X$ and $S$ is continuous, it follows that $S\equiv 0.$ Consequently, $T$ can not be a right zero-divisor.
\end{proof}

\begin{remark}\label{10}
	Proposition \ref{9} implies that $C_{\phi}\in B(\mathbb{H}^2(\mathbb{D}))$ is a right zero-divisor if and only if $\mathcal{R}(C_{\phi})$ is not dense in $\mathbb{H}^2(\mathbb{D}).$
\end{remark}
\begin{theorem}
	Let $\phi:\mathbb{D} \to\mathbb{D}$ be any analytic map which either maps $\mathbb{D}$ conformally on to a Carath$\acute{e}$odory domain or $\phi$ is  a weak* generator of $\mathbb{H}^\infty(\mathbb{D}).$ Then $C_{\phi}$ is not a right zero-divisor in $\mathfrak{B}$. 
\end{theorem}

\begin{proof}
	Let $\phi:\mathbb{D} \to\mathbb{D}$ be analytic map which maps $\mathbb{D}$ conformally on to a Carath$\acute{e}$odory domain. Then by Theorem \ref{Caratheodory}, $\mathcal{R}(C_{\phi})$ is dense in $\mathbb{H}^{2}(\mathbb{D}).$ Now, by Proposition \ref{9},  $C_{\phi}$ can not be a right zero-divisor in $\mathfrak{B}(\mathbb{H}^{2}(\mathbb{D}))$.\\ Further, if $\phi:\mathbb{D} \to\mathbb{D}$ is analytic map which is a weak* generator of $\mathbb{H}^\infty(\mathbb{D}),$ then, by Theorem \ref{weak*generator}, $\mathcal{R}(C_{\phi})$ is dense in $\mathbb{H}^{2}(\mathbb{D}).$ Then again, $C_{\phi}$ can not be a right zero-divisor in $\mathfrak{B}(\mathbb{H}^{2}(\mathbb{D}))$. 
\end{proof}

\begin{theorem}\label{anurag12}
	The composition operator $C_{\phi}\in \mathcal{B}(\mathbb{H}^p(\mathbb{D}))$ is left zero-divisor if and only if $\phi$ is a constant function.
\end{theorem}
\begin{proof}
	Let $\phi:\mathbb{D} \to \mathbb{D}$ be defined as $\phi(z)=z_0~\forall~z\in \mathbb{D}.$ Let $g_{0}(z)=z-z_0$ and $T:\mathbb{H}^p(\mathbb{D}) \to \mathbb{H}^p(\mathbb{D})$ be defined as $T(f)=g_0\cdot f .$ Now, for each $f\in \mathbb{H}^p(\mathbb{D})$ and $z\in \mathbb{D}$ 
	\begin{align*}
		C_\phi(T)(f)(z)&=(g_0\cdot f)(\phi)(z)\\&=(g_0\cdot f)(z_0)\\&=0.
	\end{align*}
  Hence $C_\phi$ is a left zero-divisor.\\
	Conversely, suppose $C_\phi$ is a left zero-divisor. Then there exists $0\neq T \in \mathcal{B}(\mathbb{H}^p(\mathbb{D}))$ such that $C_\phi(T)(f)=0 ~\forall f \in \mathbb{H}^p(\mathbb{D}).$ This implies that $$T(f)\vert_{\mathcal{R}(\phi)}=0~\forall f \in \mathbb{H}^p(\mathbb{D}).$$ Hence $\mathcal{R}(\phi)$ must be a singleton set.
\end{proof}

\begin{theorem}\label{anur11}
	If $\phi:\mathbb{D} \to \mathbb{D}$ is a constant function, then $C_{\phi}\in \mathcal{B}(\mathbb{H}^p(\mathbb{D}))$ is right zero-divisor.
\end{theorem}
\begin{proof}	
	Let $\phi:\mathbb{D} \to \mathbb{D}$ be defined as $\phi(z)=z_0~\forall~z\in \mathbb{D}$ and $f \in \mathbb{H}^p(\mathbb{D}).$ Then $f(z)=\sum_{n=0}^{\infty}a_nz^n.$
	Now, define $T:\mathbb{H}^p(\mathbb{D}) \to \mathbb{H}^p(\mathbb{D})$ as $$T(f)=\sum_{n=0}^{\infty} a_{n+1}z^n.$$ Then for each $f \in \mathbb{H}^p(\mathbb{D}),$ $$(T \circ C_\phi)(f)(z)=T((f\circ \phi)(z))=0.$$ Hence $C_\phi$ is a right zero-divisor.
\end{proof}

\begin{remark}
	The following lemma shows that the converse of the above Theorem \ref{anur11} is not true.
\end{remark}
\begin{lemma}
	For each integer $k\geq2,$ let $\phi:\mathbb{D} \to \mathbb{D}$ be defined as $\phi(z)=z^k~\forall~z\in \mathbb{D}.$ Then $C_{\phi}$ is a right zero-divisor.
\end{lemma}

\begin{proof}
	Let $\phi(z)=z^k$ and $f\in \mathbb{H}^2(\mathbb{D}).$ Then $$f(z)=\sum_{n=0}^{\infty}a_nz^n,~ \text{where}~ \sum_{n=0}^{\infty}|a_n|^2<\infty.$$ Hence, $$(C_{\phi}f)(z)=f(z^k)=\sum_{n=0}^{\infty}a_nz^{nk}.$$ Now, for $g(z)=\sum_{n=0}^{\infty}b_nz^n \in \mathbb{H}^2(\mathbb{D}),$ define $T:\mathbb{H}^2(\mathbb{D}) \to \mathbb{H}^2(\mathbb{D})$ as $$T(g)=\sum_{n=0}^{\infty} a_{nk+1}z^{nk+1}.$$ Then $T(C_{\phi}f)=0~\forall f\in \mathbb{H}^2(\mathbb{D}).$ Hence $C_{\phi}$ is a right zero-divisor.
\end{proof}

\subsection{Composition operator on $\ell^p~(1\leq p<\infty)$ space as TDZ} \label{sec3.6}\

In this section, we give a characterization of the composition operators on $\ell^p$ which are zero-divisors and topological divisors of zero in $\mathcal{B}(\ell^p).$

\begin{theorem}\label{Anurag12}
	$C_{\phi}\in \mathcal{B}(\ell^p)$ is a right zero-divisor if and only if $\phi$ is not injective.
\end{theorem}
\begin{proof}
	Suppose $\phi$ is injective. Let $T\in \mathcal{B}(\ell^p)$ be such that $T C_{\phi}=0.$ Then $$T C_{\phi}(\mychi_n)=T(\mychi_{\phi^{-1}(n)})=0 \ \text{ for each } \ n\geq1.$$ Since $\phi$ is injective, therefore $\vert\phi^{-1}(n)\vert\leq 1 \ \forall \ n\in \mathbb{N}$. Also $$\bigcup_{n=1}^{\infty}\phi^{-1}(n)=\mathbb{N}.$$ Hence, it follows that $T(\mychi_n)=0~\forall~ n\geq 1.$ This implies that $T=0.$ Hence $C_{\phi}$ can not be a right zero-divisor.\\
	Conversely, suppose $\phi$ is not injective. Then $C_\phi$ is not surjective. Hence there exists an $f_0\in \ell^p \backslash \mathcal{R}(C_\phi)$. Let $\mathfrak{B_1}=\{f_\alpha:\alpha \in \Delta_1\}$ be a Hamel basis for $\mathcal{R}(C_\phi)$ and $\mathfrak{B}$ be a Hamel basis for $\ell^p$ containing $\mathfrak{B_1}$ and $f_0.$ Now, define $T:\ell^p \to \ell^p$ as $$T(f)=\begin{cases}
		f_0, & \text{if}~ f=f_0 \\ 0, &\text{if}~ f \in \mathfrak{B}\backslash \{f_0\}. 
	\end{cases}$$ Clearly, $T$ defines a bounded linear operator on $\ell^p$ and for each $f\in \ell^p,$  $$T C_\phi(f)=T(C_\phi f)=0.$$ This implies that $C_\phi$ is a right zero-divisor.
\end{proof}
\begin{theorem}\label{anurag13}
	$C_{\phi}\in \mathcal{B}(\ell^p)$ is a left zero-divisor if and only if $\phi$ is not surjective.
\end{theorem}

\begin{proof}
	Suppose $\phi$ is surjective. Then $C_\phi$ is injective. Let $T\in \mathcal{B}(\ell^p)$ be such that $C_\phi T=0.$ Then for each $n\geq1,~C_\phi T(\mychi_n)=0.$ This implies that, $T(\mychi_n)=0~\forall~n\geq1.$ Therefore $T=0$ and hence, $C_\phi$ can not be a left zero-divisor. \\
	Conversely, suppose $\phi$ is not surjective. Then $C_\phi$ is not injective. Hence there exists $0\neq f_0\in\ell^p$ such that $C_\phi(f_0)=0.$ Define $T:\ell^p \to \ell^p$ as, $$T(f)=f(1)f_0~~\forall ~f\in \ell^p.$$ Clearly, $T\neq0$ since $T(\mychi_1)=f_0 \neq0.$ Also, for each $f\in\ell^p,$ $$\Vert Tf\Vert=\vert f(1)\vert \Vert f_0\Vert\leq\Vert f\Vert\Vert f_0\Vert.$$ Hence $T$ is bounded, and for each $f\in\ell^p,$ $C_\phi T(f)=C_\phi(f(1)f_0)=f(1)C_\phi(f_0)=0.$ Hence $C_\phi T=0.$ Therefore $C_\phi$ is a left zero-divisor.
\end{proof}

\begin{theorem}
	$C_\phi \in \mathcal{B}(\ell^p)$ is a TDZ if and only if $\phi$ is not invertible.
\end{theorem}
\begin{proof}
	If $\phi$ is invertible then, by Theorem $\ref{anurag7},$ $C_\phi$ is also invertible. Hence $C_\phi$ can not be TDZ.\\
	Conversely, suppose $\phi$ is not invertible. If $\phi$ is not injective, then by Theorem $\ref{Anurag12},$ $C_\phi$ is a right zero-divisor. Further, if $\phi$ is not surjective, then by Theorem $\ref{anurag13},$ $C_\phi$ is a left zero-divisor. Therefore $C_\phi$ is a TDZ.
\end{proof}
\begin{remark}
	Let $\phi_1$ and $\phi_2$ be self maps on $\mathbb{N}$ such that either $\phi_1$ is not surjective or $\phi_2$ is not injective, then $C_{\phi_2}\circ C_{\phi_1}$ is a TDZ in $\mathcal{B}(\ell^p).$ 
\end{remark}
\begin{proof}
	Suppose either $\phi_1$ is not surjective or $\phi_2$ is not injective, then $\phi_1\circ \phi_2$ is not invertible. Hence by above theorem, $C_{\phi_1\circ \phi_2}$ is a TDZ.
	%	 $C_{\phi_1\circ \phi_2}(f)=f\circ(\phi_1\circ \phi_2)=(f\circ\phi_2)\circ \phi_1=C_{\phi_1}(f\circ \phi_2)=C_{\phi_1}(C_{\phi_2}(f))=(C_{\phi_1}\circ C_{\phi_2})(f).$ 
	But, $C_{\phi_1\circ \phi_2}=C_{\phi_2}\circ C_{\phi_1}.$ Therefore $C_{\phi_2}\circ C_{\phi_1}$ is a TDZ.
\end{proof}
%\begin{center}
%	\textbf{Declarations} 
%\end{center}

%$\bullet$ \textbf{Funding}
%The first author is supported by the Council of Scientific and Industrial Research (CSIR) NET-JRF, New Delhi, India, through grant $09/013(0891)/2019-EMR-I.$\\

%$\bullet$ \textbf{Consent Statement} Not applicable.\\
%$\bullet$ \textbf{Data Availability Statement} Not applicable.\\
%$\bullet$\textbf{Conflicts of Interest} The author declares no conflict of interest.

\end{document}